\documentclass[12pt,a4paper]{amsart}
\usepackage{amsmath}
\usepackage{amsfonts}
\usepackage{amsthm}
\usepackage{amscd}
\usepackage{amsbsy}
\usepackage{amsrefs}
\usepackage[latin2]{inputenc}
\usepackage{t1enc}
\usepackage[mathscr]{eucal}
\usepackage{indentfirst}
\usepackage{graphicx}
\usepackage{graphics}
\usepackage{pict2e}
\usepackage{epic}
\numberwithin{equation}{section}
\usepackage[margin=2.9cm]{geometry}
\usepackage{epstopdf} 
\usepackage{tikz-cd}
\usepackage{chngcntr}
\usepackage{amssymb}
\usepackage{esint}
\usepackage{hyperref}
\usepackage{mwe}
\usepackage{mathrsfs}
\usepackage{xfrac} 
\usepackage{verbatim}

\theoremstyle{plain}
\newtheorem{Th}{Theorem}[section]
\newtheorem{Lemma}[Th]{Lemma}

 \theoremstyle{definition}
\newtheorem{Def}[Th]{Definition}

\newtheorem{Rem}[Th]{Remark}
\newtheorem{?}[Th]{Problem}

\usepackage{color}

\newcommand*\R{\mathbb{R}}

\newcommand*\Om{\Omega}

\newcommand{\Div}{\text{div}_x}
\newcommand{\vectoru}{\mathbf{u}}
\newcommand{\vectorU}{\mathbf{U}}
\newcommand{\vectorv}{\mathbf{v}}
\newcommand{\Epsilon}{\mathcal{E}}
\newcommand{\dx}{\text{ d}x}
\newcommand{\dt}{\text{ d}t\;}
\newcommand{\invdiv}{\nabla_x \Delta^{-1}_{x}}
\newcommand{\vectorphi}{\pmb{\varphi}}

\begin{document}

\title[On weak-strong uniqueness]{On weak-strong uniqueness for compressible Navier-Stokes system with general pressure laws}

\author[N. Chaudhuri]{Nilasis Chaudhuri}

\address{Institute f{\"u}r Mathematik, Technische Universit{\"a}t Berlin, Straße des 17. Juni 136, D -- 10623 Berlin, Germany.} 

\email{chaudhuri@math.tu-berlin.de}

 \subjclass[2010]{Primary: {35 Q 30} . Secondary: {35 B 30}}

 \keywords{Compressible Navier stokes equation, Weak--Strong uniqueness, Non-monotone pressure}

\begin{abstract} The goal of the present paper is to study the weak--strong uniqueness problem for the compressible Navier--Stokes system with a general barotropic pressure law. Our results include the case of a hard sphere pressure 
of Van der Waals type with a non--monotone perturbation and a Lipschitz perturbation of a monotone pressure. Although the main tool is the relative energy inequality, the results are conditioned by the presence of viscosity and do not seem extendable to the Euler system.

\end{abstract}

\maketitle

\section{Introduction} 
Let $T>0$ and $\Omega \subset \mathbb{R}^d,\; d\in \{2,3\}$ be a bounded domain. We consider the compressible Navier-Stokes equation in time-space cylinder $(0,T)\times \Omega$:
\begin{align}
\partial_t \varrho + \text{div}_x (\varrho \mathbf{u})&=0,\label{eq:cont}\\
\partial_t(\varrho \mathbf{u}) + \Div (\varrho \mathbf{u} \otimes \mathbf{u})+\nabla_x p(\varrho)&=\Div \mathbb{S}(\nabla_x \mathbf{u})\label{eq:momentum}. 
\end{align}
Here $\mathbb{S}(\nabla_x \mathbf{u})$ is \textit{Newtonian stress tensor} defined by
\begin{align}\label{cauchy_str}
\mathbb{S}(\nabla_x \mathbf{u})=\mu \bigg(\frac{\nabla_x \vectoru + \nabla_x^{T} \vectoru}{2}-\frac{1}{d} (\Div\vectoru)\mathbb{I} \bigg) + \lambda (\Div \vectoru) \mathbb{I},
\end{align}
where $\mu >0$ and $\lambda >0$ are the \textit{shear} and \textit{bulk} viscosity coefficients, respectively. An external force $\mathbf{f}$ can be included in the momentum equation \eqref{eq:momentum}. 

We focus on two basic types of boundary conditions:(i) the no slip boundary condition:
\begin{equation} \label{noslip}
\vectoru|_{\{\partial\Om \times (0,T)\}=0},
\end{equation}
or (ii) the periodic boundary conditions, where the domain $\Omega$ is identified with the flat torus,
\begin{equation} \label{period}
\Omega = \left( [-1,1]|_{\{-1, 1 \}} \right)^d.
\end{equation}
 
We denote $ C^{0,1}[0,\infty)$ the space of globally Lipschitz functions.
Here we consider the following pressure laws.

\begin{itemize}
	\item\textbf{[Barotropic Law]} In a perturbation of the isentropic setting, the pressure $p$ and the density $\varrho$ of the fluid are interrelated by :
	\begin{align}\label{p-condition}
	p(\varrho)=a\varrho^{\gamma} + q(\varrho),\; \text{with } \gamma\geq  1,\;  a>0\text{ and }  
	q\in C^{0,1}[0,\infty) \ \mbox{globally Lipschitz.}
	\end{align}
	As a matter of fact, the hypothesis on $\gamma$ will reflect the growth of $q$ as $\varrho \to \infty$.
	Note that our goal is not to show \emph{existence} of solutions but stability  of strong solutions in a larger class 
	of weak/measure valued solutions.
	
	\item \text{[General pressure]} Instead of considering $a\varrho^\gamma$ we can take a more 
	general {barotropic equation of state},
	\begin{align}\label{p- condition2}
	\begin{split}
	&p(\varrho)=h(\varrho)+q(\varrho), \text{ with } q\in C^{0,1}[0,\infty) \ \mbox{globally Lipschitz},\\
	&h\in C^1[0,\infty),\;  h(0)=0,\;h'>0,\;\text{in } (0,\infty),\; \text{and } 
	\liminf_{\varrho \rightarrow \infty} \frac{h(\varrho)}{\varrho^{\gamma}} >0  \  \text{with } \ 
	\gamma \geq 1.
	\end{split}
	\end{align}

	The above hypotheses on the equation of state are motivated by the recent work of Bresch and Jabin \cite{bresch-jabin} .
	
	\item \textbf{[Hard-Sphere Law]} {Finally, we consider a singular pressure law, where the pressure $p$ and the density $\varrho$ of the fluid are interrelated} by a non-monotone hard-sphere equation of state in the interval $[0,\bar{\varrho})$:
	\begin{align}\label{pr-defn}
	\begin{split}
	&p\in C^1[0,\bar{\varrho}),\;  p(\varrho)=h(\varrho)+ q(\varrho),h(0)=0,\\ & h^{'} >0 \; \text{on }(0,\bar{\varrho}),\;\lim_{\varrho\rightarrow \bar{\varrho}} h(\varrho)=+\infty,\; q\in C^1_{c}(0,\bar{\varrho}).
	\end{split}
	\end{align} 
\end{itemize}

Compressible Navier--Stokes system has been widely studied by many people in the last few decades. 

\begin{itemize}
	\item If $q\equiv 0$, the relation \eqref{p-condition} reduces to standard \textit{isentropic} equation of state, for which the problem \eqref{eq:cont}-\eqref{cauchy_str}  admits global in time weak solutions for any finite energy initial data, see Antontsev et al.\cite{Ant-kazhi-monakov} for $d=1$, Lions\cite{plionsbook} for $d=2$, $\gamma\geq \frac{3}{2}$, $d=3$, $\gamma \geq \frac{9}{5}$, and 
	\cite{fei-DVCF} for $d=2$, $\gamma>1$, $d=3$, $\gamma>\frac{3}{2}$. 
{One can go further and introduce the more general class of measure--valued solutions in the spirit 
of the pioneering work of	 DiPerna \cite{diperna-mv}}. Regarding compressible Navier-Stokes for $\gamma \geq 1$, Feireisl et al.\cite{Fei-piotr-agnies-weid-2016} proved the existence of dissipative measure valued solution.

	\item If $q\neq 0$, the pressure need not be a monotone function of density. A weak solution however still exists for $\gamma> \frac{3}{2}$ and $N=3$, see \cite{Fei-2002-exis-jde}. Recently instead of compactly supported $q$, Bresch-Jabin \cite{bresch-jabin} proved existence for more general pressure. 
	
	\item The pressure law \eqref{pr-defn} is motivated by two famous models for viscous fluids, namely Van Der Waal's equation of state and Hard sphere law modeled by Carnahan-Sterling. Now for some constant temperature Van Der Waal's equation of state gives,
	\begin{align*}
	p(\varrho)= C \frac{\bar{p}(\varrho)}{\bar{\varrho}-\varrho},
	\end{align*} 
	where $\bar{p}$ is some polynomial and $\bar{\varrho}$ is a positive constant. {The Carnahan-Sterling model reflects the the hard sphere model and is given by,
		\begin{align*}
		p(\varrho)= C \frac{\tilde{p}(\varrho)}{(\bar{\varrho}-\varrho)^3},
		\end{align*}
		with $\tilde{p}$ polynomial and $\bar{\varrho}$ positive constant.} The main difference for these models from isentropic setting is here $p(\varrho)\rightarrow +\infty$ when $\varrho \rightarrow \bar{\varrho}$. In \cite{fei-lu-malek2016} Feireisl et. al. and \cite{Fei-Zhang} Feireisl and Zhang, {the existence of global weak solution for similar models 
		was shown.}
\end{itemize}

The weak--strong uniqueness principle asserts that a weak and the strong solution emanating for the same initial data coincide as long as the strong solution exists. {The leading idea is based on the concept of relative entropy that goes 
back to the pioneering paper by Dafermos \cite{Daf4} , and that was later exploited in different context by 
Berthelin and Vasseur \cite{BeVa1}, Mellet and Vasseur \cite{MeVa1}, or Saint--Raymond \cite{SaiRay} to name a few examples} 

The available results for the compressible Navier--Stokes system are as follows.

\begin{itemize}

\item 
Germain \cite{Germain} showed the weak--strong uniqueness in a class of weak solutions enjoying extra regularity properties. Unfortunately, 
the existence of weak solutions in his class is still an open problem.

\item
Feireisl, Novotn\' y and Sun \cite{fei-nov-sun-2011} and Feireisl, Jin and Novotn\' y \cite{fei-jin-nov} showed the weak--strong uniqueness result 
in the existence class for a isentropic (barotropic) pressure equation of state with strictly increasing pressure. These results were extended to 
the class of the so--called dissipative masure--valued solutions by Feireisl et al. \cite{Fei-piotr-agnies-weid-2016}.

\item 
Feireisl, Lu and Novotn\' y \cite{Fei-yong-nov-2018}extended the weak--strong uniqueness principle to the hard--sphere pressure type equation of state, still with 
strictly monotone pressure--density relation.

\item 
{Recently, Feireisl \cite{fei-2018-nm} proved weak--strong uniqueness in the class of weak solutions, with a non--monotone 
compactly supported perturbation of the isentropic equation of state.}

\end{itemize}

Our goal in this paper is to extend the weak--strong uniqueness principle in two directions:

\begin{itemize}
	
	\item[1.] {To extend the results of \cite{fei-2018-nm} to a more general class of non--monotone Lipschitz perturbations.}
    \item[2.] {To consider non--monotone compact perturbations of the hard sphere model in the context of weak solutions.}

\end{itemize}

The plan for this paper is as follows:
\begin{itemize}
	\item In the first part we will discuss about weak--strong uniqueness where pressure is given by \eqref{p-condition} and \eqref{p- condition2}.
	\item For the later part of the paper we discuss weak strong uniqueness when pressure is given by \eqref{pr-defn}.
\end{itemize}

\part{Pressure following equation of state \eqref{p-condition} and \eqref{p- condition2}}

In this part, we focus on the problem with the no-slip boundary conditions \eqref{noslip}.

\section{Dissipative Weak Solution, Main Result}
Before going to our formal discussion, define \textit{pressure potential } as :\\

\begin{itemize}
	\item For $h(\varrho)=a \varrho^\gamma$, 
	\begin{align} \label{P-defn2}
	P(\varrho)= H(\varrho)+Q(\varrho), \text{ where }H(\varrho)=\frac{a}{\gamma-1} \varrho^\gamma \text{ and }Q(\varrho)=\varrho \int_{1}^{\varrho} \frac{q(z)}{z^2} \; \text{d}z .
	\end{align}
	\item When $p$ is given by the more general formula \eqref{p- condition2},
	\begin{align}\label{P-defn}
	\begin{split}
	&P(\varrho)= H(\varrho)+Q(\varrho) \text{ where}\\
	&H(\varrho)=\varrho \int_{1}^{\varrho} \frac{h(z)}{z^2} \; \text{d}z,\; Q(\varrho)=\varrho \int_{1}^{\varrho} \frac{q(z)}{z^2} \; \text{d}z.
	\end{split}
	\end{align}
\end{itemize}
As a trivial consequence of the above we obtain,
\begin{align}
	\begin{split}
		&\varrho H^{'}(\varrho) -H(\varrho)= h(\varrho) \text{ and } \varrho H^{''} (\varrho)=h'(\varrho) \text{ for } \varrho>0,\\
		&\varrho Q^{'}(\varrho) -Q(\varrho)= q(\varrho) \text{ and } \varrho Q^{''} (\varrho)=q'(\varrho) \text{ for } \varrho>0.
	\end{split}
\end{align}
We impose some hypothesis on the initial data as,
\begin{align}\label{id1}
\begin{split}
&\varrho(0,\cdot) =\varrho_{0} (\cdot),\; \varrho_{0}\geq 0\text{ a.e. in }\Om \text{ and } \varrho_{0}\in L^{\gamma}(\Om),\\
	&  \; \int_{\Omega}\bigg( \frac{\vert (\varrho \vectoru)_0 \vert^2 }{\varrho_0} +H(\varrho_0) \bigg) < \infty.
\end{split}
\end{align}
The definition of \emph{Dissipative Weak Solution} is as follows:
\begin{Def}\label{defw}
	We say that $[\varrho, \vectoru]$ is a dissipative weak solution in $(0,T)\times \Omega$ to the system of equations \eqref{eq:cont}-\eqref{cauchy_str}, with the no-slip condition \eqref{noslip}, supplemented with initial data \eqref{id1} and pressure follows the law \eqref{P-defn} and \eqref{P-defn2} if:
	
	\begin{itemize}
		\item \textbf{Regularity Class:} $\varrho \in C_{w}([0,T];L^{\gamma}(\Om))$, 
		$ p(\varrho) \in L^{1}((0,T)\times \Omega)$,\\ $\vectoru \in L^2(0,T;W^{1,2}_{0} (\Omega; \mathbb{R}^d))$, 
		$\varrho \vectoru \in C_{w}([0,T];L^{\frac{2\gamma}{\gamma + 1}}(\Om;\mathbb{R}^d))$, \\$ \varrho \vert u \vert^2 \in L^{\infty}(0,T;L^{1}(\Om))$.
		\item \textbf{Renormalized equation of Continuity:} 
		For any $\tau \in (0,T) $ and any $\varphi \in C_{c}^{1}([0,T]\times \bar{\Om})$ {it holds} 
		\begin{align}\label{continuity_eqn_weak}
		\begin{split}
		&\big[ \int_{ \Om} (\varrho+ b(\varrho)) \varphi \dx\big]_{t=0}^{t=\tau} \\
		&= 
		\int_0^{\tau} \int_{\Om} [ (\varrho+b(\varrho)) \partial_t \varphi + (\varrho+ b(\varrho)) \vectoru \cdot \nabla_x \varphi+ (b(\varrho)-\varrho b'(\varrho)\Div \vectoru \varphi] \dx \dt ,
		\end{split}
		\end{align} 
		where, $b\in C^1[0,\infty),\; \exists r_b>0 \text{ such that }b'(x)=0,\; \forall x>r_b$.
		\item \textbf{Momentum equation:}For any $\tau\in (0,T)$ and any $\pmb{\varphi} \in C^{1}_c([0,T]\times \Om;\mathbb{R}^d)$, it holds
		\begin{align}\label{momentum_eqn_weak}
		\begin{split}
		&\bigg[\int_{\Om} \varrho \vectoru(\tau,\cdot)\cdot \vectorphi(\tau,\cdot) \dx\bigg]_{t=0}^{t=\tau} \\
		&=\int_0^{\tau}\int_{\Om} [\varrho \vectoru\cdot \partial_{t} \vectorphi + (\varrho \vectoru \otimes \vectoru : \nabla_x \vectorphi + p(\varrho) \Div \vectorphi -\mathbb{S}(\nabla_x \vectoru):\nabla_x \vectorphi] \dx \dt ,\\		
		\end{split}
		\end{align}
		\item \textbf{Energy inequality:}For a.e. $\tau \in (0,T)$, the energy inequality holds:
		\begin{align}\label{eng_ineq_weak}
		\begin{split}
		\bigg[\int_{\Om} &\bigg( \frac{1}{2} \varrho \vert \vectoru \vert^2 +P(\varrho)\bigg)(t,\cdot) \dx \bigg]_{t=0}^{t=\tau}+ \int_{0}^{\tau} \int_{\Om} \mathbb{S}(\nabla_x \vectoru) :\nabla_x \vectoru\dx \dt \leq 0.\\
		\end{split}
		\end{align}
	\end{itemize}
\end{Def}

\subsection{Discussion of Definition:}
 Now from \eqref{P-defn} we have $H(\varrho) \approx \varrho^\gamma $, 
{and $Q(\varrho) \approx \varrho \log (\varrho)$ for all $\varrho$ large enough. In particular, there is a constant  
	$c > 0$ such that}
\[
P(\varrho) \geq \frac{1}{2} H(\varrho) - c \ \mbox{for all}\ \varrho \geq 0 
\]
With help of a limiting procedure in \eqref{continuity_eqn_weak}  we have,
\begin{equation} \label{pom}
\big[ \int_{\Om} b(\varrho)(t,\cdot) \dx \big]_{t=0}^{t=\tau}=-\int_{0}^{\tau} \int_{ \Om} (\varrho b'(\varrho)-b(\varrho))\Div\vectoru \dx \dt.
\end{equation}

The class of functions $b$ in \eqref{pom} can be extended to those for which
both $b'(\varrho) \varrho$ and $b(\varrho)$ belong to $L^2(0,\infty)$. {As our goal is to apply 
\eqref{pom} to the globally Lipschitz perturbation $q$, we have to assume $\gamma \geq 2$ in the pressure law. Note that similar 
hypothesis is also used by Bresch and Jabin \cite{bresch-jabin}. Accordingly, we have}  
\begin{align}
\bigg[\int_{ \Om}  Q(\varrho)  \dx \bigg]_{t=0}^{t=\tau} = - \int_{0}^{\tau} \int_{ \Om} q(\varrho)  \Div \vectoru \dx \dt
\end{align}
{as long as $q$ is a globally Lipschitz function and $\gamma \geq 2$. Consequently}
\begin{align}\label{modified_eng_ineq}
\begin{split}
\bigg[\int_{\Om} &\bigg( \frac{1}{2} \varrho \vert \vectoru \vert^2 +H(\varrho)\bigg)\dx \bigg]_{t=0}^{t=\tau}+ \int_{0}^{\tau} \mathbb{S}(\nabla_x \vectoru) :\nabla_x \vectoru\dx \dt \leq \int_{0}^{\tau} \int_{ \Om} q(\varrho)  \Div \vectoru \dx \dt.\\
\end{split}
\end{align}
\subsection{Main Result}
Our goal is to show the following result.
\begin{Th}\label{theorem1}
	Let $\Om \subset \R^d$, $d=1,2,3$, be a bounded Lipschitz domain. Let the pressure be given by \eqref{p-condition} or \eqref{p- condition2}, with $\gamma \geq 2$. Suppose that $[\varrho,\vectoru]$ is a dissipative weak solution and $[r,\vectorU]$ a classical solution of the problem \eqref{eq:cont}-\eqref{cauchy_str} with no slip boundary condition 
	{\eqref{noslip}} on the time interval $[0,T] $ such that,
	\[ \varrho(0,\cdot)=r(0,\cdot)>0,\; \varrho\vectoru(0,\cdot)=r(0,\cdot)\vectorU(0,\cdot). \]
	Then 
	\[ \varrho=r,\; \vectoru=\vectorU\text{ in }(0,T)\times \Om.\]
\end{Th}

\begin{Rem}
{Hypothesis $\gamma \geq 2$ is related to the growth of the perturbation $q$ when $\varrho \to \infty$. The result remains valid for any $\gamma \geq 1$ as soon as}
\[
q'(\varrho) \approx \varrho^\alpha \ \mbox{for}\ \varrho \to \infty,\ 
\mbox{where}\ \alpha + 1 \leq \frac{\gamma}{2}.
\]

\end{Rem}

In the next section we will prove the result.
\section{Relative Energy and Weak Strong uniqueness}
\subsection{Relative Energy}
Following \cite{fei-jin-nov} and \cite{fei-2018-nm} (cf. the standard reference material by Dafermos \cite{Daf4}) we introduce \emph{relative energy functional}:

\begin{align}\label{rel ent weak}
	\Epsilon(t)=\Epsilon(\varrho,\vectoru \vert r,\vectorU)(t):= \int_{\Omega}\frac{1	}{2} \varrho \vert \vectoru-\vectorU\vert^2 + (H(\varrho)-H(r) -H^{\prime}(r)(\varrho -r)) (t,\cdot) \dx ,
\end{align} 
where $r,\vectorU$ are arbitrary test functions and $[\rho,\vectoru]$ in \eqref{rel ent weak} is weak solution of \eqref{eq:cont}-\eqref{cauchy_str} as in \eqref{defw}.
By direct calculation we can show that,
\begin{align}
	\begin{split}
		\Epsilon(\tau) &=\int_{ \Om} \big( \frac{1}{2}\varrho \vert \vectoru \vert ^2 + H(s) \big) \dx - \int_{ \Om} \varrho \vectoru  \cdot \vectorU \dx \\
		& \quad +\int_{ \Om} \frac{1}{2} \varrho  \vert \vectorU \vert^2 \dx - \int_{ \Om}\varrho H^\prime(r) \dx + \int_{ \Om} h(r) \dx = \Sigma_{i=1}^{5} K_i
	\end{split}
\end{align}

Now we look for the terms $K_i$ for $i=1(1)5$. First we note that $K_1$ can be evaluated 
by means of \eqref{modified_eng_ineq}. To compute $K_2$ we use \eqref{momentum_eqn_weak} and for $K_3,K_4$ we use \eqref{continuity_eqn_weak}. Calculating above terms we get,
\begin{align}\label{re1}
	\begin{split}
		[\mathcal{E}(t)]_{t=0}^{t=\tau} + & \int_0^\tau \int_{ \Om} \mathbb{S}(\nabla_x \vectoru) : (\nabla_x \vectoru -\nabla_x \vectorU) \;\dx \dt \\
		\leq & - \int_0^{\tau} \int_{ \Om} \varrho \vectoru \cdot \partial_{t}\vectorU \dx \dt \\
		& - \int_0^{\tau} \int_{ \Om} [\varrho\vectoru \otimes \vectoru  : \nabla_x \vectorU + h(\varrho)\Div\vectorU ] \dx \dt\\
		&+\int_0^{\tau} \int_{ \Om} [\varrho \vectorU \cdot \partial_{t} \vectorU + \varrho \vectoru \cdot (\vectorU \cdot \nabla_x)\vectorU] \dx \dt\\
		&+\int_{0}^{\tau} \int_{ \Om}\bigg[ (1-\frac{\varrho}{r}) h^\prime(r) \partial_t r - \varrho \vectoru \cdot \frac{h^\prime(r)}{r} \nabla_x r \bigg] \dx \dt\\
		&-\int_{0}^{\tau} \int_{ \Om} q(\varrho) \Div \vectorU \dx \dt+\int_{0}^{\tau} \int_{ \Om} q(\varrho) \Div \vectoru \dx \dt\\
	\end{split}
\end{align}

Now if we assume $[r,\vectorU]$ satisfies \eqref{eq:cont}-\eqref{cauchy_str}, and these are smooth solution with $r>0$ then we have,

\begin{align}\label{re2}
	\begin{split}
		[\mathcal{E}(t)]_{t=0}^{t=\tau} + & \int_0^\tau \int_{ \Om} \mathbb{S}(\nabla_x \vectoru -\nabla_x \vectorU) : (\nabla_x \vectoru -\nabla_x \vectorU) \;\dx \dt  \\
		&\leq -\int_0^\tau \int_{ \Om} \mathbb{S}(\nabla_x \vectorU) : (\nabla_x \vectoru -\nabla_x \vectorU) \;\dx \dt \\
		& - \int_0^{\tau} \int_{ \Om} (\varrho \vectoru -\varrho \vectorU)  \cdot (-(\vectorU\cdot\nabla_x)\vectorU -\frac{1}{r}\nabla_x p(r) + \frac{1}{r} \Div \mathbb{S}(\nabla_x \vectorU ) )\dx \dt \\
		&  + \int_{0}^{\tau} \int_{ \Om}  h(\varrho) \Div\vectorU  \dx \dt\\
		&+\int_{0}^{\tau} \int_{ \Om}\bigg[ (1-\frac{\varrho}{r}) h'(r) \partial_t r - \varrho \vectoru  \cdot \frac{h'(r)}{r} \nabla_x r \bigg] \dx \dt\\
		&-\int_{0}^{\tau} \int_{ \Om}  q(\varrho)\Div \vectorU \dx \dt+\int_{0}^{\tau} \int_{ \Om} q(\varrho) \Div \vectoru \dx \dt.
	\end{split}
\end{align}

Thus we have,
\begin{align}\label{rel_ent_REM}
\begin{split}
[\mathcal{E}(t)]_{t=0}^{t=\tau} +& \int_0^\tau \int_{ \Om} \mathbb{S}(\nabla_x \vectoru- \nabla_x \vectorU) : (\nabla_x \vectoru -\nabla_x \vectorU) \;\dx \dt \\
& \leq \int_0^\tau \int_{ \Om}  (\frac{\varrho}{r}-1) (\vectorU- \vectoru)\; ( \Div \mathbb{S}(\nabla_x \vectorU)-\nabla_x q(r))\; \dx\dt\\
&  +\int_0^\tau \int_{ \Om} \varrho (\vectoru-\vectorU)\cdot ((\vectorU -\vectoru)\cdot \nabla_x)\mathbf{U} \;\dx \dt\\
&+\int_0^\tau \int_{ \Om}(- h(\varrho) +h(r)+ h'(r)(\varrho -r))\; \Div \vectorU\;   \dx \dt\\
&+ \int_0^{\tau} \int_{\Om} (\Div \vectoru -\Div \vectorU)(q(\varrho)-q(r))\; \dx \dt= \Sigma_{i=1}^{4} \mathcal{L}_i.
\end{split}
\end{align} 
Here $\mathcal{L}_i$ for $i=1(1)4$ have been termed as \textit{remainder terms}.\\

We know that for our choice of interrelation between pressure and density we have,

\begin{Lemma}\label{p by P 1}
	Suppose $H$  is defined as \eqref{P-defn} and $r$ lies on a compact subset of $(0,\infty)$ then we have,
	\begin{align}
		H(\varrho)-H(r)-H^{\prime}(r)(\varrho -r) \geq c(r)  
		\begin{cases}
			&(\varrho -r)^2 \text{ for } r_1 \leq \varrho \leq r_2,\\
			&(1+ \varrho^{\gamma})  \text{ otherwise }\\
		\end{cases},
	\end{align}
	where, $c(r)$ is uniformly bounded for $r$ belonging to compact subsets of $(0,\infty)$. 
\end{Lemma}

Hence for \eqref{p-condition} and \eqref{p- condition2}  we have,
\begin{Lemma}\label{p by P 2}
	For $\varrho\geq 0$,
	\begin{align}
		\vert h(\varrho)-h(r) -h'(r) (\varrho -r) \vert \leq C(r)(  H(\varrho)-H(r)-H'(r)(\varrho -r) ),
	\end{align}
	where $C(r)$ is uniformly bounded if $r$ lies in some compact subset of $(0,\infty)$.
\end{Lemma}
\begin{proof}
	The proof of both lemmas have been discussed in \cite{fei-nov-sun-2011} and \cite{fei-jin-nov}.
\end{proof}
\begin{Rem}\label{p byP rem}
	In our case we choose $r_1,r_2$ such that they satisfy, $r_1 < \frac{\inf\limits_{(x,t)\in (0,T)\times \Om} r(x,t)}{2}$, $ r_2 > 2 \times {\sup\limits_{(x,t)\in (0,T)\times \Om} r(x,t)}$ and $ 1 + \varrho^\gamma \geq \max\{ \varrho, \varrho^2 \},\; \forall \varrho \geq r_2$. 
\end{Rem}

\subsection{Weak strong uniqueness}
Now we want to compute remainder terms i.e. $\mathcal{L}_i$ for $i=1,2,3,4$. For $\mathcal{L}_2$ we have,
\begin{align}\label{L_2}
\begin{split}
 \int_0^\tau \int_{ \Om} \varrho (\vectoru-\vectorU)\cdot ((\vectorU -\vectoru)\cdot \nabla_x)\mathbf{U} \;\dx \dt \leq \Vert \nabla_x \vectorU \Vert_{C([0,T]\times \Om)} \int_0^\tau \mathcal{E}(t) \dt. 
\end{split}
\end{align}
Next for $\mathcal{L}_3$ we use lemma \eqref{p by P 2} and obtain,

\begin{align}\label{L_3}
\begin{split}
\int_0^\tau \int_{ \Om}(- h(\varrho) & +h(r)+ h'(r)(\varrho -r))\; \Div \vectorU\;   \dx \dt \\
& \leq \Vert \nabla_x \vectorU \Vert_{C([0,T]\times \Om)}  \int_0^\tau \mathcal{E}(t) \dt.
\end{split}
\end{align}

Now we focus on $\mathcal{L}_1$ and $\mathcal{L}_4$. Since $q$ is globally Lipschitz by \textit{Rademacher Theorem} $q$ is almost everywhere differentiable and its derivative is less than the Lipschitz constant $L_q$.
Hence we obtain,
\begin{align*}
\vert \frac{1}{r} \nabla_x q(r) \vert \leq  \frac{L_q}{\inf r}\Vert r \Vert_{C^1}
\end{align*}
Consider $\psi \in C_c^{\infty}(0,\infty),\; 0\leq \psi \leq 1,\; \psi(s)=1 \text{ for }s\in (r_1,r_2)$. 
Then we have, 
\begin{align*}
&(\varrho-r)(\vectorU-\vectorv)=\psi(\varrho)(\varrho-r)(\vectorU-\vectorv)+(1-\psi(\varrho))(\varrho-r)(\vectorU-\vectorv).
\end{align*}

Consequently, we obtain 
\begin{align}
\begin{split}
 \psi(\varrho) (\varrho-r)(\vectorU-\vectorv)  \leq \frac{1}{2} \frac{\psi^2(\varrho)}{\sqrt{\varrho}} (\varrho-r)^2  + \frac{1}{2}  \frac{\psi^2(\varrho)}{\sqrt{\varrho}} \varrho \vert \vectorU-\vectoru \vert^2.
\end{split}
\end{align}

Now using  that $\psi$ is compactly supported in $(0,\infty)$ and  lemma \eqref{p by P 1} we control both the terms by $\mathcal{E}(\cdot)$. Thus we have,
\begin{align}\label{psi}
\begin{split}
&\int_0^{\tau} \int_{ \Om}  \psi(\varrho) (\varrho-r)(\vectorU - \vectoru) \cdot \frac{1}{r}\big( \Div \mathbb{S}(\nabla_x \vectorU)- \nabla_x q(r) \big) \dx \dt \\
& \leq ( \Vert \frac{1}{r}\big( \Div \mathbb{S}(\nabla_x \vectorU) \big) \Vert_{C([0,T]\times \bar{\Om}; \R^{d})} + \frac{L_q}{\inf r}\Vert r \Vert_{C^1([0,T]\times \bar{\Om}; \R^{d})}) \int_{0}^{\tau} \Epsilon(t) \dt
\end{split}
\end{align}

We rewrite $1-\psi(\varrho)=w_1(\varrho)+w_2(\varrho)$, where $\text{supp}(w_1)\subset [0,r_1)$ and $\text{supp}(w_2)\subset (0,r_2]$,

\begin{align*}
& (1-\psi(\varrho)) (\varrho-r)(\vectorU-\vectoru) = (w_1(\varrho)+w_2(\varrho)) (\varrho-r)(\vectorU-\vectoru) .
\end{align*}

For $\delta>0$ we obtain,
\begin{align*}
w_1(\varrho) (\varrho-r)(\vectorU-\vectoru) \leq C(\delta) w_1^2(\varrho)(\varrho-r)^2   + \delta  \vert \vectorU - \vectoru \vert^2 .
\end{align*}
Thus using Poincar{\'e} inequality we have,
\begin{align}\label{w_1}
\begin{split}
&\int_0^{\tau} \int_{ \Om}  w_1(\varrho) (\varrho-r)(\vectorU - \vectorv) \cdot \frac{1}{r}\big( \Div \mathbb{S}(\nabla_x \vectorU)- \nabla_x q(r) \big) \dx \dt\\
&\leq C( \Vert \frac{1}{r}\big( \Div \mathbb{S}(\nabla_x \vectorU) \big) \Vert_{C([0,T]\times \bar{\Om}; \R^{d})} + \frac{L_q}{\inf r}\Vert r \Vert_{C^1([0,T]\times \bar{\Om}; \R^{d})}, \delta) \int_{0}^{\tau} \Epsilon(t) \dt\\
&\quad + \delta \int_0^\tau \int_{ \Om} \mathbb{S}(\nabla_x \vectoru -\nabla_x \vectorU) : (\nabla_x \vectoru -\nabla_x \vectorU) \;\dx \dt  .
\end{split}
\end{align}
Next,
\begin{align*}
w_2(\varrho) (\varrho-r)(\vectorU-\vectorv) \leq C(r) (\varrho   + \varrho  \vert \vectorU - \vectoru \vert^2) .
\end{align*}
Using remark of Lemma \eqref{p byP rem} we obtain
\begin{align}\label{w_2}
\begin{split}
&\int_0^{\tau} \int_{ \Om}  w_2(\varrho) (\varrho-r)(\vectorU - \vectoru) \cdot \frac{1}{r}\big( \Div \mathbb{S}(\nabla_x \vectorU)- \nabla_x q(r) \big) \dx \dt \\
& \leq C ( \Vert \frac{1}{r}\big( \Div \mathbb{S}(\nabla_x \vectorU) \big) \Vert_{C([0,T]\times \bar{\Om}; \R^{d})} + \frac{L_q}{\inf r}\Vert r \Vert_{C^1([0,T]\times \bar{\Om}; \R^{d})}) \int_{0}^{\tau} \Epsilon(t) \dt
\end{split}
\end{align}

So combining \eqref{psi}, \eqref{w_1} and \eqref{w_2} we obtain

\begin{align}\label{L_1}
\begin{split}
&\int_0^{\tau} \int_{ \Om}   (\varrho-r)(\vectorU - \vectorv) \cdot \frac{1}{r}\big( \Div \mathbb{S}(\nabla_x \vectorU)- \nabla_x q(r) \big) \dx \dt\\
&\leq C(\delta,r,\vectorU,q) \int_{0}^{\tau} \Epsilon(t) \dt + \delta \int_0^\tau \int_{ \Om} \mathbb{S}(\nabla_x \vectoru -\nabla_x \vectorU) : (\nabla_x \vectoru -\nabla_x \vectorU) \;\dx \dt  .
\end{split}
\end{align}

Next is term $\mathcal{I}_4$,

\begin{align*}
&\int_0^{\tau} \int_{\Om} (\Div \vectoru -\Div \vectorU)(q(\varrho)-q(r))\; \dx \dt \\& \leq  C(\delta) L_q \int_0^{\tau} \int_{\Om} (\varrho-r)^2 \dx \dt + \delta \int_0^{\tau} \int_{\Om} \vert \Div \vectorU - \Div \vectoru \vert^2 \dx\dt .
\end{align*}
Note that, by virtue of our choice of the no--slip 
boundary conditions \eqref{noslip} the last integral is controlled by
\[
\int_0^\tau \int_{ \Om} \mathbb{S}(\nabla_x \vectoru -\nabla_x \vectorU) : (\nabla_x \vectoru -\nabla_x \vectorU) \;\dx \dt.
\]

Using a similar argument as in the earlier case we can say that,
\begin{align}\label{L_4}
\begin{split}
&\int_0^{\tau} \int_{ \Om}   (\Div \vectoru -\Div \vectorU)(q(\varrho)-q(r))\; \dx \dt\\
&\leq C(\delta,r,q) \int_0^\tau \mathcal{E}(t) \dt + \delta \int_0^\tau \int_{ \Om} \mathbb{S}(\nabla_x \vectoru -\nabla_x \vectorU) : (\nabla_x \vectoru -\nabla_x \vectorU) \;\dx \dt .
\end{split}
\end{align}
Thus combining \eqref{L_2}, \eqref{L_3}, \eqref{L_1} and \eqref{L_4} and choosing $\delta$ small we obtain,
\begin{align}\label{re final 1}
[\mathcal{E}(t)]_{t=0}^{t=\tau} +& 
\frac{1}{2} \int_0^\tau \int_{ \Om} \mathbb{S}(\nabla_x \vectoru- \nabla_x \vectorU) : (\nabla_x \vectoru -\nabla_x \vectorU) \;\dx \dt\leq C(r,\vectorU,q) \int_{0}^{\tau} \mathcal{E}(t)\dt.
\end{align}

\subsection{End of the proof}
\begin{proof}[Proof of Theorem \eqref{theorem1}]
	As $C(r,\vectorU,q)$ in \eqref{re final 1} is uniformly bounded in $[0,T]$. Hence we apply Gr{\"o}nwall's inequality and using hypothesis on initial condition we obtain $\mathcal{E}=0$ a.e. in $(0,T)$.
\end{proof}

\part{Pressure following equation of state \eqref{pr-defn}}

{In this part, we focus on the problem endowed with the periodic boundary conditions \eqref{period}. Accordingly, 
the domain $\Omega$ is here and hereafter identified with the flat torus 
$\Omega = \left([-1,1]|_{\{ -1, 1 \}} \right)^d$.}
Now considering density-pressure interrelation \eqref{pr-defn} we will define weak solution and study the weak-strong uniqueness.

	\section{Dissipative Weak solution, Main result}
	
	We impose some hypothesis on the initial data as,
	\begin{align}\label{id}
	\begin{split}
	&\varrho(0,\cdot) =\varrho_{0} (\cdot) \text{ with } 0\leq \varrho_{0} < \bar{\varrho} \text{ in }
	\Om, \; \int_{\Om} H(\varrho_0) \dx < \infty,\\
	& \vectoru(0,\cdot) =\vectoru_{0} (\cdot), \; \int \frac{\vert u_0 \vert^2 }{\varrho_0} < \infty.
	\end{split}
	\end{align}
	Weak solution are defined as follows:
	\begin{Def}\label{def1}
		We say that $(\varrho, \vectoru)$ is a dissipative weak solution in $(0,T)\times \Omega$ to the system of equations 
		(\ref{eq:cont}--\ref{cauchy_str}), with the periodic boundary conditions \eqref{period}, supplemented with initial data \eqref{id}, if:
		\begin{itemize}
			\item $0\leq \varrho < \bar{\varrho}$ \text{ a.e. in }$(0,T)\times\Om$, $\varrho \in C_{w}([0,T];L^{\gamma}(\Om))$ for any $\gamma > 1$,\\ $ p(\varrho) \in L^{1}((0,T)\times \Omega)$, {$\vectoru \in L^2(0,T;W^{1,2} (\Omega; \mathbb{R}^d))$}, $\varrho \vectoru \in C_{w}([0,T];L^{2}(\Om;\mathbb{R}^d))$, $ \varrho \vert u \vert^2 \in L^{\infty}(0,T;L^{1}(\Om))$. 
			\item For any $\tau\in (0,T)$ and any test function $\varphi \in C^{\infty}([0,T]\times {\Om})$, one has
			\begin{align}
			\int_0^{\tau} \int_{\Om} [ \varrho \partial_t \varphi + \varrho \vectoru \cdot \nabla_x \varphi] \dx \dt = \int_{\Om} \varrho(\tau,\cdot) \varphi(\tau,\cdot) -\int_{\Om} \varrho_0 \varphi(0,\cdot)\dx.
			\end{align} 
			\item For any $\tau\in (0,T)$ and any test function $\pmb{\varphi} \in C^{\infty}([0,T]\times \Om;\mathbb{R}^d)$, one has
			\begin{align}
			\begin{split}
			\int_0^{\tau} &\int_{\Om} [\varrho \vectoru\cdot \partial_{t} \vectorphi + (\varrho \vectoru \otimes \vectoru : \nabla_x \vectorphi + p(\varrho) \Div \vectorphi -\mathbb{S}(\nabla_x \vectoru):\nabla_x \vectorphi] \dx \dt \\
			&= \int_{\Om} \varrho \vectoru(\tau,\cdot)\cdot \vectorphi(\tau,\cdot) \dx- \int_{\Om} \varrho_{0} \vectoru_{0} \cdot \vectorphi(0,\cdot) \dx.
			\end{split}
			\end{align}
		
			\item The continuity equation also holds in the sense of renormalized solutions:
			\begin{align}\label{renm_cont_eqn_weak}
			\begin{split}
			&\big[ \int_{ \Om} ( b(\varrho)) \varphi \dx\big]_{t=0}^{t=\tau} \\
			&= 
			\int_0^{\tau} \int_{\Om} [ b(\varrho) \partial_t \varphi +  b(\varrho) \vectoru \cdot \nabla_x \varphi+ (b(\varrho)-\varrho b'(\varrho)\Div \vectoru \varphi] \dx \dt ,
			\end{split}
			\end{align} 
			where, $\varphi \in C^{\infty}([0,T]\times {\Om})$for any $b\in C^1[0,\bar{\varrho})$ satisfying
			\begin{align}
			\vert b(s) \vert^2 + \vert b'(s)\vert^2 \leq C(1 + h(s)) \text{ for some constant } C \text{ and any } s\in [0,\bar{\varrho}).
			\end{align}
			\item For a.e. $\tau \in (0,T)$, the energy inequality holds:
			\begin{align}
			\begin{split}
			\int_{\Om} &\bigg[ \frac{1}{2} \varrho \vert \vectoru \vert^2 +P(\varrho)\bigg](\tau,\cdot) \dx + \int_{0}^{\tau} 
			{\int_{\Om}}
			\mathbb{S}(\nabla_x \vectoru) :\nabla_x \vectoru\dx \dt\\
			&\leq \int_{\Om}\bigg[ \frac{1}{2} \varrho_0 \vert \vectoru_0 \vert^2 + P(\varrho_0)\bigg] \dx ,
			\end{split}
			\end{align}
			where $P$ is given by,
			\begin{align} \label{pr potential defn}
			P(\varrho)=\varrho \int_{\frac{\bar{\varrho}}{2}}^{\varrho} \frac{p(z)}{z^2} \; \text{d}z.
			\end{align}
		\end{itemize}
	\end{Def}

\begin{Rem}
	We denote,
	 \begin{align*}
 H(\varrho)=\varrho \int_{\frac{\bar{\varrho}}{2}}^{\varrho} \frac{h(z)}{z^2} \; \text{d}z\text{ and }Q(\varrho)=\varrho \int_{\frac{\bar{\varrho}}{2}}^{\varrho} \frac{q(z)}{z^2} \; \text{d}z.
	\end{align*}
\end{Rem}
\begin{Rem}
	Throughout our discussion we have the following assumption near $\bar{\varrho}$:
	\begin{equation}\label{assumption near rhobar}
	\lim\limits_{\varrho\rightarrow \bar{\varrho}} h(\varrho)(\bar{\varrho}-\varrho)^{\beta} >0, \quad \text{ for some } \beta > \frac{5}{2}.
	\end{equation}
	This assumption is possibly technical but necessary for the analysis. Note that similar 
	assumption also ensures global existence, cf. Feireisl, Lu, Novotn\' y \cite{fei_lu_nov_hs}.
\end{Rem}

\subsection{Discussion of Definition}
Since $s\in \text{supp} (q)= [s_0, s_1]\subset (0, \bar{\varrho})$ implies,
\begin{align*}
\vert Q'(s) \vert ^2 + \vert Q (s) \vert^2 \leq \frac{1}{s_0} \vert Q(s) -q(s) \vert^2 + \vert Q(s) \vert^2 \leq M( \sup_{(0, \bar{\varrho})} \vert q \vert, \sup_{(0, \bar{\varrho})} \vert q' \vert ).
\end{align*}
Hence $Q$ satisfies hypothesis of $b$ as in \eqref{def1}. Thus from renormalized equation we have,
\begin{align}\label{q1}
\bigg[\int_{\Omega} Q(\varrho) (t,\cdot) \dx \bigg]_{t=0}^{t=\tau} = - \int_0^\tau\int_{ \Om} q(\varrho) \Div \vectoru \; \dx  \dt.
\end{align}
Since $[\varrho, \vectoru]$ is a renormalized 
dissipative weak solution, we obtain,
\begin{align*}
&\bigg[\int_{\Omega} \big(\frac{1}{2} \varrho \vert \vectoru \vert^2 + H(\varrho) \big)(t,\cdot) \dx \bigg]_{t=0}^{t=\tau} + \int_0^\tau \int_{ \Om} \mathbb{S}(\nabla_x \vectoru) : \nabla_x \vectoru \;\dx \dt \\
&\quad \; \leq  \int_0^\tau\int_{ \Om} q(\varrho) \Div \vectoru \; \dx  \dt.
\end{align*}
\begin{Rem}
	Instead of $ \vert b(\varrho) \vert^2 $ and $\vert b'(\varrho) \vert^2$ the same calculation can be done for $ \vert b(\varrho) \vert^{\frac{5}{2}} $ and $\vert b'(\varrho) \vert^{\frac{5}{2}}$.
\end{Rem}
\subsection{Main Result}
Here we state the main theorem,
\begin{Th}\label{maintheo}
		Let $\Om \subset \R^d$, $d=1,2,3$, be domain with periodic boundary condition. Let the pressure be given by \eqref{pr-defn}. Let $[\varrho,\vectoru]$ be finite energy weak solution  in the sense of definition \eqref{def1} which satisfies the pressure condition \eqref{assumption near rhobar} for some $\beta\geq3$. Let $[ r,\vectorU]$ be a classical solution of the 
	\textcolor{red}{\eqref{eq:cont}--\eqref{cauchy_str}, \eqref{period}}, i.e. $(r,\vectorU) \in C^{1}([0,T]\times \Om) \times C^1([0,T]; C^2(\Om))$ solves equation with same initial data as $(\varrho,\vectoru)$ and $0< r < \bar{\varrho}$. Then there holds,
	\begin{equation}
	(\varrho, \vectoru)= (r,\vectorU) \text{ in } (0,T)\times \Om
	\end{equation} 
\end{Th}
The next two sections are devoted to the proof of Theorem \eqref{maintheo}.

	\section{Relative Energy}
	We define the relative energy functional:
	\begin{align}\label{ent_func}
	\Epsilon(t)=\Epsilon(\varrho,\vectoru \vert r,\vectorU)(t):= \int_{\Omega}\frac{1	}{2} \varrho \vert \vectoru-\vectorU\vert^2 + (H(\varrho)-H(r) -H^{'}(r)(\varrho -r)) (t,\cdot) \dx .
	\end{align} 
	We rewrite the entropy functional as 
	\begin{align*}
	\mathcal{E}(t)= &\int_{\Omega} \big(\frac{1}{2} \varrho \vert \vectoru \vert^2 + H(\varrho) \big) \dx - \int_{\Om} \varrho \vectoru \cdot \vectorU \dx \\& 
	\hspace{5mm}+ \int_{\Om} \varrho \big(\frac{1}{2}\vert \vectorU \vert^2 -H^{'}(r)\big) \dx + \int_{\Om} \big( rH^{'}(r)-H(r) \big) \dx.\\
	\end{align*}
Next we follow the similar lines as in Section 3. We assume $[r,\vectorU]\in C^{1}([0,T]\times \Om)\times C^{1}(0,T; C^{2}(\Om))$ is classical solution of \eqref{eq:cont}-\eqref{cauchy_str}, \eqref{period} and pressure law \eqref{pr-defn} with  $0<r<\bar{\varrho}$. Hence we obtain,
		
	\begin{align*}
	[\mathcal{E}(t)]_{t=0}^{t=\tau} +& \int_0^\tau \int_{ \Om} \mathbb{S}(\nabla_x \vectoru- \nabla_x \vectorU) : (\nabla_x \vectoru -\nabla_x \vectorU) \;\dx \dt \\
	& \leq \int_0^\tau \int_{ \Om}  (\frac{\varrho}{r}-1) (\vectorU- \vectoru)\; ( \Div \mathbb{S}(\nabla_x \vectorU)-\nabla_x q(r))\; \dx\dt\\
	&  +\int_0^\tau \int_{ \Om} \varrho (\vectoru-\vectorU)\cdot ((\vectorU -\vectoru)\cdot \nabla_x)\mathbf{U} \;\dx \dt\\
	&+\int_0^\tau \int_{ \Om}(- h(\varrho) +h(r)+ h'(r)(\varrho -r))\; \Div \vectorU\;   \dx \dt\\
	&+ \int_0^{\tau} \int_{\Om} (\Div \vectoru -\Div \vectorU)(q(\varrho)-q(r))\; \dx \dt.
	\end{align*} 
	Next, we state a lemma which indicates the difference in the proof of Theorem \eqref{maintheo} with Theorem \eqref{theorem1}. 
	\begin{Lemma}\label{Lemma pressure bound}
		Let $\varrho \geq 0$ and $ 0 < \alpha_0\leq r \leq \bar{\varrho}-\alpha_0 < \bar{\varrho}$. There exists $\alpha_1\in(0,\alpha_0)$ and a constant $c>0$, such that
		\begin{align}
		H(\varrho)-H(r)-H'(r)(\varrho -r)\geq\begin{cases}
		c(\varrho -r )^2,\; &\text{if } \alpha_1 \leq \varrho \leq \bar{\varrho} -\alpha_1,\\
		\frac{h(r)}{2},\; &\text{if }
		0\leq \varrho \leq \alpha_1 ,\\
		\frac{H (\varrho)}{2},\; &\text{if }\bar{\varrho}-\alpha_1 \leq \varrho < \bar{\varrho}.
		\end{cases}
		\end{align}
		We also have,
		\begin{align}
		h(\varrho)-h(r)-h'(r)(\varrho -r)\leq\begin{cases}
		c(\varrho -r )^2,\; &\text{if } \alpha_1 \leq \varrho \leq \bar{\varrho} -\alpha_1,\\
		1+h'(r)r -h(r),\; &\text{if }
		0\leq \varrho \leq \alpha_1 ,\\
		2h(\varrho),\; &\text{if }\bar{\varrho}-\alpha_1 \le \varrho < \bar{\varrho}.
		\end{cases}
		\end{align}
	\end{Lemma}
	\begin{Rem}
		Without loss of generality we can assume on $[\bar{\varrho} -\alpha_1, \bar{\varrho})$, $H(\varrho)>2$. 
	\end{Rem}
\begin{Rem}
	Further we consider $\alpha_1$ such that $supp(q)\subset (\alpha_1,\bar{\varrho}-\alpha_1)$. 
\end{Rem}
\begin{proof}[Proof of Lemma \eqref{Lemma pressure bound}]
	The proof has been discussed in \cite{Fei-yong-nov-2018}.
\end{proof}

\subsection{Discussion of Lemma \eqref{Lemma pressure bound}}
	In Lemma \eqref{Lemma pressure bound} the constant $c$ depends on $r$ such that $c(r)$ is unformly bounded on $(\alpha_0, \bar{\varrho}-\alpha_0)$.
	From the above two results and hypotheses of \ref{Lemma pressure bound} and \ref{pr-defn} we can say that for $0\leq \varrho \leq \bar{\varrho}-\alpha_1$ we have  
	\begin{align*}
	\vert  h(\varrho)-h(r)-h'(r)(\varrho -r) \vert \leq C(H(\varrho)-H(r)-H'(r)(\varrho -r)).
	\end{align*}
    As $\bar{\varrho} - \alpha_1 \leq \varrho <\bar{\varrho}$ we have no control on $  h(\varrho)-h(r)-h'(r)(\varrho -r)$ by $ H(\varrho)-H(r)-H'(r)(\varrho -r)$, so we need to add one extra term $\int_{0}^{\tau} \int_{ \Om} b(\varrho) h(\varrho) \dx \dt $  on the left hand side of the equation which takes care of that case where $b$ is a function which satisfies the hypothesis of renormalized equation.\\
Let the symbol $\Delta_x$ denote the Laplace operator defined on spatially periodic functions with zero mean.
\subsection{Relative energy inequality with extra term}
Motivated from discussion above we rewrite the expression for relative entropy as 
	\begin{align}\label{relative entropy modified}
	\begin{split}
	[\mathcal{E}(t)]_{t=0}^{t=\tau} +& \int_0^\tau \int_{ \Om} \mathbb{S}(\nabla_x \vectoru- \nabla_x \vectorU) : (\nabla_x \vectoru -\nabla_x \vectorU) \;\dx \dt +  \int_0^{\tau} \int_{\Om} b(\varrho)h(\varrho)\\
	&\leq \int_0^\tau \mathcal{R}_1(t)\dt + \int_0^\tau \mathcal{R}_2 (t)\dt + \mathcal{R}_{3}(\tau),
	\end{split}
	\end{align}
	where, $\mathcal{R}_1 (\cdot) $ is given by
	\begin{align}
	\begin{split}
	\mathcal{R}_1(t)=& \int_{ \Om}  (\frac{\varrho}{r}-1) (\vectorU- \vectoru)\; ( \Div \mathbb{S}(\nabla_x \vectorU)-\nabla_x q(r))\; \dx\\
	&  + \int_{ \Om} \varrho (\vectoru-\vectorU)\cdot ((\vectorU -\vectoru)\cdot \nabla_x)\mathbf{U} \;\dx \\
	&+\int_{ \Om}(- h(\varrho) +h(r)+ h'(r)(\varrho -r))\; \Div \vectorU\;   \dx \\
	&+  \int_{\Om} (\Div \vectoru -\Div \vectorU)(q(\varrho)-q(r))\; \dx \\
	&= \Sigma_{i=1}^{4} \mathcal{I}_i,
	\end{split}
	\end{align}
	$\mathcal{R}_2(\cdot)$ is given by
	\begin{align}
	\begin{split}
	\mathcal{R}_2(t)=& \int_{\Om} h(\varrho) \langle b(\varrho)\rangle \dx - \int_{\Om} (q(\varrho)- q(r)) b(\varrho)\dx 
	\\ &+ \int_{\Om} q(r) b(\varrho) \dx - \int_{\Om} q(r) \langle b(\varrho) \rangle \dx + \int_{\Om} q(\varrho) \langle b(\varrho)
	\rangle \dx\\
	&- \int_{\Om} \varrho \vectoru \otimes \vectoru :\nabla_x( \invdiv (b(\varrho)-\langle b(\varrho) \rangle)) \dx\\
	& +\int_{\Om} \mathbb{S}(\nabla_x \vectoru ) : \nabla_x( \invdiv (b(\varrho)-\langle b(\varrho) \rangle)) \dx\\
	& +\int_{\Om} \varrho \vectoru \cdot \invdiv  \Div (b(\varrho)\vectoru ) \dx \\
	& +\int_{\Om} \varrho \vectoru \cdot \invdiv \big((b'(\varrho)\varrho -b(\varrho)) \Div\vectoru -\langle (b'(\varrho)\varrho -b(\varrho) \Div \vectoru\rangle\big) \dx\\
	&= \Sigma_{i=5}^{13} \mathcal{I}_i,
	\end{split}
	\end{align}
	and $\mathcal{R}_3(\cdot)$ is given by
	\begin{align}\label{r3}
	\begin{split}
	\mathcal{R}_{3} (\tau) =& \int_{\Om} \varrho \vectoru \cdot \invdiv(b(\varrho)-\langle b(\varrho)\rangle(\tau,\cdot) \dx\\
	&- \int_{\Om} \varrho_0 \vectoru_{0} \cdot \invdiv (b(\varrho_0)-\langle b(\varrho_0)\rangle \dx\\
	&= \Sigma_{i=14}^{15} \mathcal{I}_i.
	\end{split}
	\end{align}

	We sum up the above results and state the theorem,
	\begin{Th}
		Suppose the pressure constraint \eqref{assumption near rhobar} is satisfied. Let $\{ \varrho,\vectoru \}$ be a finite energy weak solution in $(0,T)\times \Omega$ in the sense of definition \eqref{def1}. Let $(r,\vectorU)\in C^{1}([0,T]\times \Om)\times C^{1}(0,T; C^{2}(\Om))$ such that 
		\begin{align*}
		0<r<\bar{\varrho}.
		\end{align*} 
		Let $b(s)\in C^1[0,\varrho)$ satisfy the condition,
		\begin{align}\label{modified assumption on b,b'}
		\vert b'(s) \vert^{\frac{5}{2}} +\vert b(s) \vert^{\frac{5}{2}} \leq C(1+h(s)) \text{ for some constant }C \text{ and any } s\in[0,\bar{\varrho}).
		\end{align}\label{modified reative energy 2}
		Then the following relative energy  true for a.e. $\tau\in (0,T)$,
		\begin{align}
		\begin{split}
		[\mathcal{E}(t)]_{t=0}^{t=\tau} +& \int_0^\tau \int_{ \Om} \mathbb{S}(\nabla_x \vectoru- \nabla_x \vectorU) : (\nabla_x \vectoru -\nabla_x \vectorU) \;\dx \dt +  \int_0^{\tau} \int_{\Om} b(\varrho)h(\varrho)\\
		&\leq \int_0^\tau \mathcal{R}_1(t)\dt + \int_0^\tau \mathcal{R}_2 (t)\dt + \mathcal{R}_{3}(\tau),
		\end{split}
		\end{align}
		with $\mathcal{R}_1$, $\mathcal{R}_2$ and $\mathcal{R}_3$ defined as above.
	\end{Th}
\begin{Rem}
	Condition \eqref{modified assumption on b,b'} is slightly changed from the similar assumption in definition \eqref{def1}. Although $b$ following \eqref{modified assumption on b,b'} will be a suitable candidate for $b$ as prescribed in definition \eqref{def1}.
\end{Rem}
\begin{proof}
 We have to check that all integrals in R.H.S of \eqref{modified reative energy 2} are bounded which is already done in \cite{Fei-yong-nov-2018}.
 
\end{proof}
	\section{Weak-strong uniqueness}
	We have achieved the derivation of remainder terms. 
	Now we consider a fixed $b$ which satisfies \eqref{modified assumption on b,b'}. Then we want to show that $\mathcal R_{i}(\cdot)$ can be bounded by $\eta(\cdot) \mathcal{E}(\cdot)$ for some positive function $\eta$, for each $i=1,2,3$.

\subsection{Choice of $b$ and its properties}
Consider $b\in C^{\infty}[0,\bar{\varrho})$, $b'(s)\geq 0$ as follows:
 \begin{align}\label{b}
 \begin{split}
 b(s)=\begin{cases}
 0 &\text{ if } s\leq \bar{\varrho} -\alpha_1,\\
 -\log(\bar{\varrho} -s ),& \text{ if } \bar{\varrho}-\alpha_2,
 \end{cases}
 \quad b'(s) >0 \text{ if } \bar{\varrho}-\alpha_1 < s < \bar{\varrho}- \alpha_2.
 \end{split}
 \end{align}
The choice of $\alpha_2$ is in such a way that 
\begin{equation}\label{alphatwo choice}
-\log(\bar{\varrho}- s) \geq 16 \Vert \Div \vectorU \Vert , \text{ if } \bar{\varrho}- \alpha_2 \leq s < \bar{\varrho}.
\end{equation}
Considering the assumption \eqref{assumption near rhobar} along with \eqref{pr potential defn} we have the following results:
\begin{equation}
\text{For } \gamma >0, \lim_{s\rightarrow \bar{\varrho}-} \frac{h(s)}{(b(s))^{\gamma}} = \lim_{s\rightarrow \bar{\varrho}-} \frac{H(s)}{(b(s))^{\gamma}} =\lim_{s\rightarrow \bar{\varrho}-} \frac{h(s)}{(b'(s))^{\beta}} = \lim_{s\rightarrow \bar{\varrho}-} \frac{H(s)}{(b(s))^{\beta-1}}= +\infty.
\end{equation}
This along with \eqref{Lemma pressure bound} yields the following, for $\gamma \geq 1$:
\begin{align}\label{b by pr potential}
\begin{split}
\int_{ \Om} \vert b(\varrho) \vert^{\gamma} \dx =& \int_{ \varrho \geq \bar{\varrho} -\alpha_1} \vert b(\varrho) \vert^{\gamma} \dx \\& \leq C \int_{ \varrho \geq \bar{\varrho} -\alpha_1} H(\varrho) \dx \\ &
\leq C \int_{ \varrho \geq \bar{\varrho} -\alpha_1} (H(\varrho)- H(r) -H'(r)(\varrho -r)) \dx.
\end{split}
\end{align}
Also for any $2\leq \beta_0 \leq \beta$, we have,
\begin{align}\label{b' by pr potential}
\begin{split}
\int_{ \Om} \vert b'(\varrho) \vert ^{\beta_0 -1} \dx &\leq C \int_{ \varrho \geq \bar{\varrho} -\alpha_1} H(\varrho) \dx \leq  C \int_{ \varrho \geq \bar{\varrho} -\alpha_1} (H(\varrho)- H(r) -H'(r)(\varrho -r)) \dx,\\
{\int_{ \Om} \vert b'(\varrho) \vert ^{\beta_0 } \dx }&{\leq C \int_{ \Om}h(\varrho)\dx }.
\end{split}
\end{align}  
\subsection{Estimates for remainder }
Now we proceed to estimate the remainder terms, As earlier mentioned, in \cite{Fei-yong-nov-2018} Feireisl, Lu, Novotn\' y have encountered similar problem with $q\equiv 0$. In a similar way we can compute terms other than $\mathcal{I}_1, \mathcal{I}_4$ and $ \mathcal{I}_6$ to $\mathcal{I}_9$.  First, 
 \begin{align}
\mathcal{I}_1 = \int_{ \Om}  (\frac{\varrho}{r}-1) (\vectorU- \vectoru)\; ( \Div \mathbb{S}(\nabla_x \vectorU)-\nabla_x q(r))\; \dx = \Sigma_{i=1}^{3} \mathcal{J}_i,
 \end{align}
with
\begin{align}
\begin{split}
&\mathcal{J}_1 := \int_{ \alpha_1\leq \varrho \leq \bar{\varrho} -\alpha_1}\int_{ \Om}  (\frac{\varrho}{r}-1) (\vectorU- \vectoru)\; ( \Div \mathbb{S}(\nabla_x \vectorU)-\nabla_x q(r))\; \dx\\
&\mathcal{J}_2:= \int_{ \varrho \leq \alpha_1} \int_{ \Om}  (\frac{\varrho}{r}-1) (\vectorU- \vectoru)\; ( \Div \mathbb{S}(\nabla_x \vectorU)-\nabla_x q(r))\; \dx\\
&\mathcal{J}_3 := \int_{ \varrho \geq \bar{\varrho} -\alpha_1}\int_{ \Om}  (\frac{\varrho}{r}-1) (\vectorU- \vectoru)\; ( \Div \mathbb{S}(\nabla_x \vectorU)-\nabla_x q(r))\; \dx.
\end{split}
\end{align}

Clearly, by Taylor's formula, Cauchy-Schwarz inequality and Poincar{\'e} inequality, we have for any $\sigma>0$,
\begin{align}
\begin{split}
\vert \mathcal{J}_1 \vert \leq \frac{C}{\sigma} \Vert & (\Div \mathbb{S}(\nabla_x \vectorU)(t)- \nabla_x q(r)) \Vert_{L^\infty} \int_{ \alpha_1\leq \varrho \leq \bar{\varrho} -\alpha_1} (\varrho - r)^2 \dx \\
& + \sigma \int_{ \Om} \vert \nabla_x (\vectorU- \vectoru) \vert^2 \dx.
\end{split}
\end{align}
Further using Korn's inequality and \eqref{Lemma pressure bound} we deduce that, 
\begin{align}
\begin{split}
\vert \mathcal{J}_1 \vert \leq \frac{C}{\sigma} \Vert & (\Div \mathbb{S}(\nabla_x \vectorU)(t)- \nabla_x q(r)) \Vert_{L^\infty} \int_{ \alpha_1\leq \varrho \leq \bar{\varrho} -\alpha_1} H(\varrho)-H(r)-H'(r)(\varrho -r) \dx \\
& + \sigma \int_{ \Om}  \mathbb{S}(\nabla_x (\vectorU- \vectoru) ):\nabla_x (\vectorU- \vectoru) \dx \\
&\leq \frac{1}{\sigma} \eta(t)\mathcal{E}(t) +  \sigma \int_{ \Om}  \mathbb{S}(\nabla_x (\vectorU- \vectoru) ):\nabla_x (\vectorU- \vectoru) \dx. \\
\end{split}
\end{align}
Similarly, 
\begin{align}
\begin{split}
\vert \mathcal{J}_2 \vert \leq \frac{C}{\sigma} \Vert & (\Div \mathbb{S}(\nabla_x \vectorU)(t)- \nabla_x q(r)) \Vert_{L^\infty} \int_{ \varrho \leq \alpha_1} 1 \dx \\
& + \sigma \int_{ \Om}  \mathbb{S}(\nabla_x (\vectorU- \vectoru) ):\nabla_x (\vectorU- \vectoru) \dx \\
\leq \frac{C}{\sigma} \Vert & (\Div \mathbb{S}(\nabla_x \vectorU)(t)- \nabla_x q(r)) \Vert_{L^\infty} \int_{ \varrho \leq \alpha_1} h(r) \dx \\
& + \sigma \int_{ \Om}  \mathbb{S}(\nabla_x (\vectorU- \vectoru) ):\nabla_x (\vectorU- \vectoru) \dx \\
&\leq \frac{1}{\sigma} \eta(t)\mathcal{E}(t) +  \sigma \int_{ \Om}  \mathbb{S}(\nabla_x (\vectorU- \vectoru) ):\nabla_x (\vectorU- \vectoru) \dx ,\\
\end{split}
\end{align}
and
\begin{align}
\begin{split}
\vert \mathcal{J}_3 \vert \leq \frac{C}{\sigma} \Vert & (\Div \mathbb{S}(\nabla_x \vectorU)(t)- \nabla_x q(r)) \Vert_{L^\infty} \int_{ \varrho \geq \bar{\varrho} -\alpha_1} 1 \dx \\
& + \sigma \int_{ \Om}  \mathbb{S}(\nabla_x (\vectorU- \vectoru) ):\nabla_x (\vectorU- \vectoru) \dx \\
\leq \frac{C}{\sigma} \Vert & (\Div \mathbb{S}(\nabla_x \vectorU)(t)- \nabla_x q(r)) \Vert_{L^\infty} \int_{ \varrho \geq \bar{\varrho} -\alpha_1} H(\varrho) \dx \\
& + \sigma \int_{ \Om}  \mathbb{S}(\nabla_x (\vectorU- \vectoru) ):\nabla_x (\vectorU- \vectoru) \dx \\
&\leq \frac{1}{\sigma} \eta(t)\mathcal{E}(t) +  \sigma \int_{ \Om}  \mathbb{S}(\nabla_x (\vectorU- \vectoru) ):\nabla_x (\vectorU- \vectoru) \dx .
\end{split}
\end{align}
{Combining above estimates, we have}
\begin{align}
\begin{split}\label{I_1}
\vert \mathcal{I}_{1} \vert \leq  \frac{3}{\sigma} \eta(t)\mathcal{E}(t) +  3\sigma \int_{ \Om}  \mathbb{S}(\nabla_x (\vectorU- \vectoru) ):\nabla_x (\vectorU- \vectoru) \dx .
\end{split}
\end{align}

Next for $\mathcal{I}_2$ we have,
\begin{align}\label{I_2}
\int_{ \Om} \varrho (\vectoru-\vectorU)\cdot ((\vectorU -\vectoru)\cdot \nabla_x)\mathbf{U} \;\dx \leq \Vert \nabla_x \vectorU(t) \Vert_{L^\infty} \int_{ \Om} \varrho \vert \vectoru - \vectorU \vert ^2 \dx \leq \eta(t) \mathcal{E}(t).
\end{align}
Looking at $\mathcal{I}_3=\Sigma_{i=4}^{5}\mathcal{J}_i$ we obtain
\begin{align}
\begin{split}
&\mathcal{J}_4 = \int_{ \varrho \leq \bar{\varrho} -\alpha_1}(- h(\varrho) +h(r)+ h'(r)(\varrho -r))\; \Div \vectorU\;   \dx ,\\
&\mathcal{J}_5 = \int_{ \varrho \geq \bar{\varrho} -\alpha_1}(- h(\varrho) +h(r)+ h'(r)(\varrho -r))\; \Div \vectorU\;   \dx .
\end{split}
\end{align} By remark of \eqref{Lemma pressure bound} we obtain 
\begin{equation}
\vert \mathcal{J}_4 \vert \leq \Vert \nabla_x \vectoru \Vert_{L^\infty} \int_{ \varrho \leq \bar{\varrho} -\alpha_1}  H(\varrho)-H(r)-H'(r)(\varrho -r) \dx \leq \eta(t) \mathcal{E}(t).
\end{equation}
For $\mathcal{J}_5$ we have to estimate carefully. We have $\alpha_2$ from \eqref{alphatwo choice}, then by \eqref{Lemma pressure bound} and \eqref{b by pr potential} give us,

\begin{align}
\begin{split}
\vert \mathcal{J}_5 \vert & \leq \int_{ \bar{\varrho}-\alpha_1\leq \varrho \leq \bar{\varrho} -\alpha_2}(- h(\varrho) +h(r)+ h'(r)(\varrho -r))\; \Div \vectorU\;   \dx \\
& \quad + \int_{ \varrho \geq \bar{\varrho} -\alpha_2} (- h(\varrho) +h(r)+ h'(r)(\varrho -r))\; \Div \vectorU\;   \dx \\
&\leq \int_{ \bar{\varrho}-\alpha_1\leq \varrho \leq \bar{\varrho} -\alpha_2} \vert \Div \vectorU \vert \max_{\bar{\varrho}-\alpha_1\leq s \leq \bar{\varrho} -\alpha_2} h''(s) (\varrho- r)^2 \dx + \frac{1}{8} \int_{ \Om} h(\varrho) b(\varrho) \dx\\
&\leq \eta(t) \mathcal{E}(t) + \frac{1}{8} \int_{ \Om} h(\varrho) b(\varrho) \dx\\
\end{split}
\end{align}  
\begin{align}\label{I_3}
\begin{split}
\vert \mathcal{I}_{3} \vert \leq \frac{1}{8}\int_{ \Om} b(\varrho)h(\varrho)\dx + \eta(t)\mathcal{E}(t) .
\end{split}
\end{align}

For $\mathcal{I}_4=\Sigma_{i=6}^{8}\mathcal{J}_i$ we write,
\begin{align}
\begin{split}
&\mathcal{J}_6=\int_{ \alpha_1\leq \varrho \leq \bar{\varrho} -\alpha_1} (\Div \vectoru -\Div \vectorU)(q(\varrho)-q(r))\; \dx, \\
&\mathcal{J}_7=\int_{ \varrho \leq \alpha_1} (\Div \vectoru -\Div \vectorU)(q(\varrho)-q(r))\; \dx ,\\
&\mathcal{J}_8=\int_{ \varrho \geq \bar{\varrho} -\alpha_1}(\Div \vectoru -\Div \vectorU)(q(\varrho)-q(r))\; \dx .\\
\end{split}
\end{align}
Now using  Cauchy-Schwarz inequality, Korn inequality and lemma \eqref{Lemma pressure bound} we have,
\begin{align}
\begin{split}
\vert \mathcal{J}_6 \vert &\leq  \frac{C}{\sigma} M \int_{ \alpha_1\leq \varrho \leq \bar{\varrho} -\alpha_1} (\varrho - r)^2 \dx + \sigma \int_{ \Om}  \mathbb{S}(\nabla_x (\vectorU- \vectoru) ):\nabla_x (\vectorU- \vectoru) \dx \\
&\leq \frac{C}{\sigma}\int_{ \alpha_1\leq \varrho \leq \bar{\varrho} -\alpha_1} H(\varrho)-H(r)-H'(r)(\varrho -r) \dx \\
&\quad  + \sigma \int_{ \Om}  \mathbb{S}(\nabla_x (\vectorU- \vectoru) ):\nabla_x (\vectorU- \vectoru) \dx \\
&\leq  \frac{1}{\sigma} \eta(t)\mathcal{E}(t) +  \sigma \int_{ \Om}  \mathbb{S}(\nabla_x (\vectorU- \vectoru) ):\nabla_x (\vectorU- \vectoru) \dx .
\end{split}
\end{align}
Since $q$ is a compactly supported function, using a similar argument we have,
\begin{align}
\begin{split}
\vert \mathcal{J}_7 \vert + \vert \mathcal{J}_8 \vert  &\leq \frac{C}{\sigma} \int_{ \varrho \leq \alpha_1} 1 \dx + \frac{C}{\sigma} \int_{ \varrho \geq \bar{\varrho} -\alpha_1} 1 \dx \\
&\quad + 2\sigma \int_{ \Om}  \mathbb{S}(\nabla_x (\vectorU- \vectoru) ):\nabla_x (\vectorU- \vectoru) \dx \\
&\leq  \frac{1}{\sigma} \eta(t)\mathcal{E}(t) +  2\sigma \int_{ \Om}  \mathbb{S}(\nabla_x (\vectorU- \vectoru) ):\nabla_x (\vectorU- \vectoru) \dx .
\end{split}
\end{align}
From the above estimates we have,
\begin{align}
\begin{split}\label{I_4}
\vert \mathcal{I}_{4} \vert \leq  \frac{1}{\sigma} \eta(t)\mathcal{E}(t) +  3\sigma \int_{ \Om}  \mathbb{S}(\nabla_x (\vectorU- \vectoru) ):\nabla_x (\vectorU- \vectoru) \dx .
\end{split}
\end{align}

Next we will focus on the terms of $\mathcal{R}_2$. 
As a consequence of \eqref{b by pr potential} we have 
\begin{align}
\langle b(\varrho) \rangle=\dfrac{1}{\vert \Om \vert} \int_{ \Om} b(\varrho) \dx\leq C \mathcal{E}(t)
\end{align}
Using above relation in $\mathcal{I}_5$, we obtain,
\begin{equation}\label{I_5}
\vert \mathcal{I}_5 \vert \leq C \mathcal{E}(t) \int_{ \Om}h(\varrho) \dx \leq \eta(t) \mathcal{E}(t).
\end{equation}
Similarly using that $q$ has compact support, we have 
\begin{align}\label{I_7,8,9}
\vert \mathcal{I}_7 \vert +\vert \mathcal{I}_8 \vert+\vert \mathcal{I}_9 \vert \leq \eta(t) \mathcal{E}(t).
\end{align}
For $\mathcal{I}_6$ we rewrite it as
\begin{align}
\vert \mathcal{I}_6 \vert \leq C \int_{ \Om} (q(\varrho)-q(r))^2 \dx + C \int_{ \Om} (b(\varrho))^2 \dx \leq \Sigma_{i=9}^{11} \mathcal{J}_i + \eta(t)\mathcal{E}(t),
\end{align}
where,
\begin{align}
\begin{split}
&\mathcal{J}_9=\int_{ \alpha_1\leq \varrho \leq \bar{\varrho} -\alpha_1} (q(\varrho)-q(r))^2\; \dx, \\
&\mathcal{J}_{10}=\int_{ \varrho \leq \alpha_1} (q(\varrho)-q(r))^2\; \dx ,\\
&\mathcal{J}_{11}=\int_{ \varrho \geq \bar{\varrho} -\alpha_1}(q(\varrho)-q(r))^2\; \dx .\\
\end{split}
\end{align}
It is similar to  $\mathcal{J}_6$, $\mathcal{J}_7$ and $\mathcal{J}_8$. Thus using a similar aregument we have,
\begin{equation}\label{I_6}
\vert \mathcal{I}_6 \vert \leq  \frac{1}{\sigma} \eta(t)\mathcal{E}(t) +  2\sigma \int_{ \Om}  \mathbb{S}(\nabla_x (\vectorU- \vectoru) ):\nabla_x (\vectorU- \vectoru) \dx .
\end{equation}

The estimates below directly follows from \cite{Fei-yong-nov-2018} with minor modification. The condition $\beta\geq 3$ plays a crucial role here.
We have, 
\begin{align}\label{I_10,11,12}
\begin{split}
\vert \mathcal{I}_{10} \vert + \vert \mathcal{I}_{11} \vert + \vert \mathcal{I}_{12} \vert \leq  \frac{1}{\sigma} \eta(t)\mathcal{E}(t) +  3\sigma \int_{ \Om}  \mathbb{S}(\nabla_x (\vectorU- \vectoru) ):\nabla_x (\vectorU- \vectoru) \dx .
\end{split}
\end{align}

Also
\begin{align}\label{I_13}
\begin{split}
\vert \mathcal{I}_{13} \vert \leq \frac{1}{8}\int_{ \Om} b(\varrho)h(\varrho)\dx + \frac{1}{\sigma} \eta(t)\mathcal{E}(t) +  2\sigma \int_{ \Om}  \mathbb{S}(\nabla_x (\vectorU- \vectoru) ):\nabla_x (\vectorU- \vectoru) \dx .
\end{split}
\end{align}

Now for initial data $\varrho_{0}=r_0 \in [\alpha_0,\bar{\varrho} -\alpha_0]$, we have $b(\varrho_{0})\equiv0$. Hence, $\mathcal{I}_{15}=0$. We can write,
\begin{align}\label{I_14,15}
\begin{split}
\vert \mathcal{I}_{14} \vert+\vert \mathcal{I}_{15} \vert \leq \frac{1}{4}\int_{ \Om} \varrho \vert \vectoru-\vectorU \vert^2 (\tau,\cdot) \dx + \frac{1}{2} \int_{ \varrho \geq \bar{\varrho} -\alpha_1} (  H(\varrho)-H(r)-H'(r)(\varrho -r) )\dx.
\end{split}
\end{align}

Combining \eqref{I_1},\eqref{I_2},\eqref{I_3}, \eqref{I_4},\eqref{I_5},\eqref{I_7,8,9} and \eqref{I_6}-\eqref{I_14,15} and choosing $\sigma$ small we can conclude that

\begin{align}\label{re_final2}
\begin{split}
[\mathcal{E}(t) ]_{t=0}^{t=\tau}+& \int_0^\tau \int_{ \Om} \mathbb{S}(\nabla_x \vectoru- \nabla_x \vectorU) : (\nabla_x \vectoru -\nabla_x \vectorU) \;\dx \dt \\ & +  \int_0^{\tau} \int_{\Om} b(\varrho)h(\varrho) \dx \dt\\
&\leq \int_{0}^{\tau}  \eta(t) \Epsilon(t) \dt,
\end{split}
\end{align}
where $\eta \in L^1(0,T)$.
\subsection{End of the proof}
\begin{proof}[Proof of Theorem \eqref{maintheo}:]
	Since $b\geq 0$, as a consequence of Gr{\"o}nwall's lemma and hypothesis for same initial data in \eqref{maintheo} we have $\mathcal{E}\equiv 0$ in $[0,T]$, which ensures our desired weak strong uniqueness result.
\end{proof}
\section{Concluding remarks}
This method cannot be extended to the Euler (inviscid) system as the viscous damping plays a crucial role in the proof.


\centerline{ \bf Acknowledgement} 

The work was supported by Einstein Stiftung, Berlin.


\end{document}